\documentclass[11pt,a4paper]{article}

\usepackage{booktabs}
\usepackage{amssymb,latexsym}
\usepackage{a4}
\usepackage[all]{xy}

\newif\iflabel
\labelfalse
\newcommand{\Label}[1]{\iflabel\ifmmode\makebox[0pt][l]{[#1]}
                       \else\marginpar{[#1]}
                       \fi\fi\label{#1} }

\newcommand{\U}{{\mathcal U}}

\newcommand{\UK}{U\!K}

\newcommand{\bs}{\backslash}
\newcommand{\ra}{\rangle}
\newcommand{\la}{\langle}

\newcommand{\X}{{\mathcal X}}
\newcommand{\Q}{{\mathcal Q}}
\newcommand{\R}{{\mathcal R}}

\newcommand{\Z}{{\mathbb Z}}
\newcommand{\N}{{\mathbb N}}
\newcommand{\rk}{\mathrm{rk}}

\newcommand{\End}{{\mathrm{End}}}
\newcommand{\Stab}{{\mathrm{Stab}}}
\newcommand{\mc}{\mathcal}

\newcommand{\Gap}{{\scshape Gap}}
\newcommand{\GAP}{\Gap}
\newcommand{\Grig}{{\mathfrak G}}
\newcommand{\FG}{\Gamma}

\newcommand{\Sym}{{\mathrm{Sym}}}
\newcommand{\Magma}{{\scshape Magma}}
\newcommand{\Rom}[1]{\MakeUppercase{\romannumeral #1}}

\newenvironment{proof}{\par\vskip-\lastskip\vskip\topsep
\noindent{\it Proof.}\vadjust{\nobreak}\quad
\begingroup\divide\topsep3\divide\itemsep3
\divide\partopsep3\divide\parskip3
\divide\parsep3}
{\ifvmode\penalty10000\hbox to\hsize{\hfil$\Box$}
\else\parfillskip0pt\widowpenalty10000\hfil$\Box$
\fi\par\vskip 1.5ex\endgroup}

\newtheorem{theorem}{Theorem}[section]
\newtheorem{proposition}[theorem]{Proposition}

\newtheorem{definition}[theorem]{Definition}

\title{{Investigating self-similar groups using their finite $L$-presentation}}
\author{Ren\'e Hartung}
\date{April 2012}

\begin{document}
\maketitle
\begin{abstract}
   Self-similar groups provide a rich source of groups with interesting
   properties; e.g., infinite torsion groups (Burnside groups) and
   groups with an intermediate word growth. Various self-similar groups
   can be described by a recursive (possibly infinite) presentation,
   a so-called finite $L$-presentation. Finite $L$-presentations allow
   numerous algorithms for finitely presented groups to be generalized
   to this special class of recursive presentations. We give an overview
   of the algorithms for finitely $L$-presented groups. As applications,
   we demonstrate how their implementation in a computer algebra system
   allows us to study explicit examples of self-similar groups including
   the Fabrykowski-Gupta groups. Our experiments yield detailed insight
   into the structure of these groups.\bigskip

\end{abstract}

%%%%%%%%%%%%%%%%%%%%%%%%%%%%%%%%%%%%%%%%%%%%%%%%%%%%%%%%%%%%%%%%%%%%%%%%%%%%
\section{Introduction}
The general Burnside problem is among the most influential problems in
combinatorial group theory. It asks whether a finitely generated group is
finite if every element has finite order. The general Burnside problem
was answered negatively by Golod~\cite{Gol64}. The first explicit
counter-examples were constructed in~\cite{Al72,Gri80,GS83}. Among
these counter-examples is the Grigorchuk group $\Grig$ which
is a finitely generated self-similar group. The group $\Grig$
is not finitely presented~\cite{Gri99} but it admits a recursive
presentation which could be described in finite terms using the
action of a finitely generated monoid of substitutions acting on
finitely many relations~\cite{Lys85}. These recursive presentations
are nowadays known as \emph{finite $L$-presentations}~\cite{Gri99} (or
\emph{endomorphic presentations}~\cite{Bar03}) in honor of Lys{\"e}nok's
work in~\cite{Lys85} for the Grigorchuk group; see~\cite{Bar03} or
Section~\ref{sec:SelfSim} for a definition.\smallskip

Finite $L$-presentations allow computer algorithms to be employed
in the investigation of the groups they define. A first algorithm
for finitely $L$-presented groups is the nilpotent quotient
algorithm~\cite{Har08,BEH08}. Recently, further algorithms for finitely
$L$-presented groups were developed~\cite{Har10,Har11,Har11b}. For
instance, in~\cite{Har11}, a coset enumeration process for finitely
$L$-presented groups was described. This is an algorithm which, given a
finite generating set of a subgroup of a finitely $L$-presented group,
computes the index of the subgroup in the finitely $L$-presented group
provided that this index is finite. Usually index computations in
self-similar groups have involved lots of tedious calculations (e.g.,
finding an appropriate quotient of the self-similar group; computing
the index of the subgroup in this quotient; followed by a proof that
the obtained index is correct; see, for instance,~\cite[Section~4]{BG02}
or~\cite[Chapter~\Rom{8}]{Har00}). The coset enumerator in~\cite{Har11}
makes this process completely automatic and thus it shows the
significance of finite $L$-presentations in the investigation of
self-similar groups. Moreover, coset enumeration allows one to
compute the number of low-index subgroups of finitely $L$-presented
groups~\cite{Har11}.\smallskip

We demonstrate the application of the algorithms for finitely
$L$-presented groups in the investigation of a class of self-similar
groups $\FG_p$ for $3 \leq p \leq 11$. The group $\FG_3$ was introduced
in~\cite{FG85}.  It is a self-similar group with an intermediate word
growth~\cite{FG85,FG91,BP09}. The groups $\FG_p$, with $p>3$, were
introduced in~\cite{Gri00}. They are known as \emph{Fabrykowski-Gupta
groups}. Their abelianization $\FG_p / \FG_p' \cong \Z_p \times \Z_p$
was computed in~\cite{Gri00}. Moreover, for $p\geq 5$, the groups $\FG_p$
are just-infinite, regular branch groups~\cite{Gri00}. The congruence
subgroups of $\FG_p$, for primes $p>3$, were studied in~\cite{Su07}; see
also~\cite{FAZR11}.  The lower central series sections $\gamma_c\FG_3 /
\gamma_{c+1}\FG_3$ have been computed entirely in~\cite{Bar05} while,
for $p>3$, parts of the lower central series sections $\gamma_c\FG_p /
\gamma_{c+1} \FG_p$ have been computed in~\cite{BEH08}. So far, little
more is known on the groups $\FG_p$.\smallskip

For $p \geq 3$, the Fabrykowski-Gupta group $\FG_p$ admits a finite
$L$-presentation~\cite{BEH08}. We demonstrate how the implementations
of the algorithms for finitely $L$-pre\-sen\-ted groups allow us to
investigate the groups $\FG_p$ for $3\leq p\leq 11$ in detail. For
instance, we demonstrate the application of our algorithm
\begin{itemize}\addtolength{\itemsep}{-1ex}
\item to compute the isomorphism type of the lower central series sections 
      $\gamma_c\FG_p / \gamma_{c+1} \FG_p$ using improved (parallel) methods
      from~\cite{BEH08,Har08}.
\item to compute the isomorphism type of the Dwyer quotients $M_c(\FG_p)$
      of their Schur multiplier using the methods from~\cite{Har10}.
\item to determine the number of low-index subgroups of the groups $\FG_p$ 
      using the methods from~\cite{Har11}.
\item to compute the isomorphism type of the sections $\FG_p^{(c)} /
      \FG_p^{(c+1)}$ of the derived series combining the methods
      from~\cite{Har11b} and~\cite{BEH08,Har08}.
\end{itemize}
We briefly sketch the algorithms available for finitely $L$-presented groups.
Moreover, we compare our experimental results
for the Fabrykowski-Gupta groups $\FG_p$ with those results for
the Grigorchuk group $\Grig$. The group $\Grig$ has been investigated
for decades now. Even though a lot is known about its structure, various questions still remain open~\cite{Gr05}. For further details
on the Grigorchuk group $\Grig$, we refer to~\cite[Chapter~\Rom{8}]{Har00}.

%%%%%%%%%%%%%%%%%%%%%%%%%%%%%%%%%%%%%%%%%%%%%%%%%%%%%%%%%%%%%%%%%%%%%%%%%%%%
\section{Self-Similar Groups}\Label{sec:SelfSim}
A self-similar group can be defined by its recursive action on a regular
rooted tree: Consider the $d$-regular rooted infinite tree ${\mc T}_d$ as a free
monoid over the alphabet $\X = \{0,\ldots,d-1\}$. Then a self-similar
group can be defined as follows:
\begin{definition}
  A group $G$ acting faithfully on the free monoid $\X^*$ is 
  \emph{self-similar} if for each $g \in G$ and $x \in \X$ there exist
  $h\in G$ and $y \in \X$ so that
  \begin{equation}
    (xw)^g = y\,w^h \quad\textrm{for each}\quad w \in \X^*.\Label{eqn:SelfSimAct}
  \end{equation}
\end{definition}
It suffices to specify the self-similar action in
Eq.~(\ref{eqn:SelfSimAct}) on a generating set of a group. For
instance, the Grigorchuk group $\Grig = \la a,b,c,d\ra$ can be defined
as a subgroup of the automorphism group of the rooted binary tree ${\mc
T}_2 = \{0,1\}^*$ by its self-similar action:
\[
  \begin{array}{rcl@{\hspace{3cm}}rcl}
  (0\,w)^a&=&1\,w     &(1\,w)^a&=&0\,w\\
  (0\,w)^b&=&0\,w^a   &(1\,w)^b&=&1\,w^c\\
  (0\,w)^c&=&0\,w^a   &(1\,w)^c&=&1\,w^d\\
  (0\,w)^d&=&0\,w     &(1\,w)^d&=&1\,w^b\,.
  \end{array}
\]
The Fabrykowski-Gupta group $\FG_3$ is another example of a self-similar
group. It was introduced in~\cite{FG85} as a group with
an intermediate word growth~\cite{FG91,BP09}. The group $\FG_3$ was
generalized in~\cite{Gri00} to a class of self-similar groups $\FG_d$
acting on the $d$-regular rooted tree:
\begin{definition}
  For $d \geq 3$, the \emph{Fabrykowski-Gupta group} $\FG_d = \la a,r \ra$
  is a self-similar group acting faithfully on the $d$-regular rooted tree
  ${\mc T}_d = \{0,\ldots,d-1\}^*$ by 
  \[
    \begin{array}{rcll}
     (x\,w)^a&=&x+1\pmod d\,w,&\textrm{for }0\leq x\leq d-1\\[0.75ex]
     (0\,w)^r&=&0\,w^a,          &                          \\
     (x\,w)^r&=&x\,w,            &\textrm{for }1\leq x< d-1\\
     (d-1\,w)^r&=&d-1\,w^r.      &                          
    \end{array}
  \]
\end{definition}
The groups $\Grig$ and $\FG_d$ admit a finite $L$-presentation;
that is, a \emph{finite $L$-presentation} is a group
presentation of the form
\begin{equation}
   \Big\la \X~\Big|~\Q \cup \bigcup_{\sigma \in \Phi^*} \R^\sigma \Big\ra,\Label{eqn:LPresApp}
\end{equation}
where $\X$ is a finite alphabet, $\Q$ and $\R$ are finite subsets of
the free group $F$ over $\X$, and $\Phi^*$ denotes the monoid of
endomorphisms which is generated by the finite set $\Phi\subseteq \End(F)$. The
group defined by the finite $L$-presentation in Eq.~(\ref{eqn:LPresApp})
is denoted by \mbox{$\la\X\mid\Q\mid\Phi\mid\R\ra$}. If $\Q = \emptyset$ holds, 
the $L$-presentation in Eq.~(\ref{eqn:LPresApp}) is \emph{ascending}. In
this case, every endomorphism $\sigma \in \Phi^*$ induces an endomorphism
of the group $G$.\smallskip

The Grigorchuk group $\Grig$ is an example of a self-similar group
which is finitely $L$-presented~\cite{Lys85}: the group $\Grig$ satisfies
\[
  \Grig \cong \Big\la \{a,b,c,d\}\:\Big|\:\{a^2,b^2,c^2,d^2,bcd\} \cup \bigcup_{i\geq 0}\,
  \{ (ad)^4, (adacac)^4 \}^{\sigma^i}\Big\ra,
\]
where $\sigma$ is the endomorphism of the free group $F$ over
$\{a,b,c,d\}$ which is induced by the map $a\mapsto aca$, $b\mapsto
d$, $c\mapsto b$, and $d\mapsto c$.  A general method for computing a
finite $L$-presentation for a class of self-similar groups was developed
in~\cite{Bar03} in order to prove
\def\0{Bartholdi~\cite{Bar03}}
\begin{theorem}[\0]\Label{thm:FinLPres}
  Each finitely generated, contracting, semi-fractal regular branch group
  is finitely $L$-presented; however, it is not finitely presented.
\end{theorem}
The constructive proof of Theorem~\ref{thm:FinLPres} in~\cite{Bar03}
was used in~\cite{BEH08} to compute the following finite $L$-presentation
for the Fabrykowski-Gupta group $\FG_p$:
\def\0{Bartholdi et al.~\cite{BEH08}}
\begin{theorem}[\0]\Label{thm:FGLp}
  For $d\geq 3$, the group $\FG_d$ is finitely
  $L$-presented by $\la\{\alpha,\rho\} \mid \emptyset\mid \{\varphi\}\mid
  \R \ra$ where the iterated relations in $\R$ are defined as follows:
  Writing $\sigma_i = \rho^{\alpha^i}$, for $1 \leq i \leq d-1$, and reading
  indices modulo $d$, we have
  \[
    {\mathcal R}=\left\{
    \alpha^d, 
    \Big[\sigma_i^{\sigma_{i-1}^k},\sigma_j^{\sigma_{j-1}^\ell}\Big],
    \sigma_i^{-\sigma_{i-1}^{k+1}} 
    \sigma_i^{\sigma_{i-1}^k\sigma_{i-1}^{\sigma_{i-2}^\ell}}
    \right\}_{  
%   \begin{array}{c}
      {\scriptstyle 1\leq i,j\leq d},\:
      {\scriptstyle 2\leq |i-j|\leq d-2},\:
      {\scriptstyle 0\leq k,\ell \leq d-1}
%   \end{array}
    }
  \]
  The substitution $\varphi$ is induced by the map $\alpha\mapsto
  \rho^{\alpha^{-1}}$ and $\rho\mapsto \rho$.
\end{theorem}
It follows immediately from the $L$-presentation in Theorem~\ref{thm:FGLp}
that the substitution $\varphi$ induces an endomorphism of the group
$\FG_d$.  Finite $L$-presentations $\la\X\mid\Q\mid\Phi\mid\R\ra$
whose substitutions $\sigma\in\Phi$ induce endomorphisms of the
group are \emph{invariant $L$-presentations}. Each ascending
$L$-presentation is invariant. It is also easy to see that
the $L$-presentation for the Grigorchuk group $\Grig$ above is
invariant~\cite[Corollary~4]{Gri98}.\smallskip

A finite $L$-presentation allows us to define a group that is possibly
infinitely presented in computer algebra systems such as \Gap~\cite{GAP}
or \Magma~\cite{Magma}. Beside defining a self-similar group by its
finite $L$-presentation, it can also be defined by its recursive action
on a regular tree.  A finite approximation of the recursive action of a
self-similar group is often sufficient to study finite index subgroups
since various self-similar groups have the congruence property: every
finite index subgroup contains a level stabilizer (i.e., the stabilizer of
some level of the regular tree). This often yields an alternative approach
to investigate the structure of a self-similar group with the help of
computer algebra systems~\cite{FR}. However, there are self-similar
groups that do not have the congruence property~\cite{BS10}. For these
groups, their finite $L$-presentation may help to gain insight into
the structure of the group. The groups $\Grig$ and $\FG_3$ have the
congruence property~\cite{BG02}.\smallskip

In the following, we demonstrate how the finite $L$-presentation
in Theorem~\ref{thm:FGLp} allows us to obtain detailed information
on the structure of the groups $\FG_p$, for $3 \leq p \leq 11$.
For further details on self-similar groups, we refer to the monograph
by Nekrashevych~\cite{Nek05}.

%%%%%%%%%%%%%%%%%%%%%%%%%%%%%%%%%%%%%%%%%%%%%%%%%%%%%%%%%%%%%%%%%%%%%%%%%%%%
\section{A Nilpotent Quotient Algorithm}\Label{sec:NQL}
For a group $G$, the lower central series is defined recursively by
$\gamma_1 G = G$ and $\gamma_{c+1} = [\gamma_cG,G]$ for $c\in\N$. If
$G$ is finitely generated, $G / \gamma_{c+1}G$ is polycyclic
and therefore it can be described by a polycyclic presentation; i.e., a polycyclic presentation is a finite
presentation whose generators refine a subnormal series with cyclic
sections. A polycyclic presentation allows effective computations within
the group it defines~\cite[Chapter~9]{Sims94}.\smallskip

A nilpotent quotient algorithm computes a polycyclic presentation
for the factor group $G/ \gamma_{c+1}G$ together with a homomorphism
$G \to G/\gamma_{c+1} G$. Such an algorithm for finitely presented
groups was developed in~\cite{Nic96}. This nilpotent quotient
algorithm was a first algorithm that could be generalized to finite
$L$-presentations~\cite{Har08,BEH08}. The experimental results in
this section were obtained with an improved, parallel version of the
algorithm in~\cite{Har08,BEH08}. They extend the computational results
in~\cite{BEH08} significantly.\smallskip

We briefly sketch the nilpotent quotient algorithm for
finitely $L$-presented groups in the following. Let $G =
\la\X\mid\Q\mid\Phi\mid\R\ra$ be a finitely $L$-presented group. Denote
by $F$ the free group over the alphabet $\X$ and let $K$ be the
normal closure $K = \left\la \bigcup_{\sigma\in\Phi^*} \R^\sigma
\right\ra^F$. First, we assume that $\Q = \emptyset$ holds. Then
$K^\sigma \subseteq K$, for each $\sigma\in\Phi$, and $G = F/K$
hold. Therefore, each $\sigma\in\Phi$ induces an endomorphism of the
group $G$. Furthermore, we have $G / \gamma_cG \cong F / K\gamma_cF$.
The nilpotent quotient algorithm uses an induction on $c$ to compute a
polycyclic presentation for $G / \gamma_cG$.  For $c = 2$, we have
\[
  G / [G,G] \cong F/KF'
  \cong (F/F')/(KF'/F').
\]
Since $G$ is finitely generated, $F/F'$ is free abelian with finite
rank. The normal generators $\bigcup_{\sigma\in\Phi^*} \R^\sigma$
of $K$ give a (possibly infinite) generating set of $KF'/F'$. From
this generating set it is possible to compute a finite generating set
${\mc U}$ with a spinning algorithm. The finite generating set ${\mc U}$
allows us to apply the methods from~\cite{Nic96} that eventually compute
a polycyclic presentation for $F/KF'$ together with a homomorphism $F
\to F/KF'$ which induces $G \to G/G'$.\smallskip

For $c>2$, assume that the algorithm has already computed a polycyclic
presentation for $G/ \gamma_cG \cong F / K\gamma_cF$ together with a
homomorphism $F \to F / K\gamma_cF$. Consider the factor group $H_{c+1}
= F / [K\gamma_cF, F]$. Then $[K\gamma_cF,F] = [K,F]\gamma_{c+1}F$
and $H_{c+1}$ satisfies the short exact sequence
\[
  1 \to 
  K\gamma_cF / [K\gamma_cF, F] \to  H_{c+1} \to
% F / [K\gamma_cF,F] \to 
  F / K\gamma_c F \to 1;
\]
that is, $H_{c+1}$ is a central extension of a finitely generated abelian
group by $G/\gamma_cG$.  Thus $H_{c+1}$ is nilpotent and polycyclic.
A polycyclic presentation for $H_{c+1}$ together with a homomorphism
$F \to F / [K\gamma_cF,F]$ can be computed with the covering algorithm
in~\cite{Nic96}; for a proof that this algorithm generalizes to finite
$L$-presentations we refer to~\cite{Har08}. Then $K\gamma_{c+1}F /
[K\gamma_c,F]$ is subgroup of $K\gamma_cF / [K\gamma_cF, F]$ and a
(possibly infinite) generating set for $K\gamma_{c+1}F / [K\gamma_cF,F]$
can be obtained from the normal generators of $K$.  Again, a finite
generating set ${\mc U}$ for $K\gamma_{c+1}F/[K\gamma_cF,F]$ can
be computed with a spinning algorithm from the normal generators of
$K$. The finite generating set $\U$ allows us to apply the methods
in~\cite{Nic96} for computing a polycyclic presentation for $G /
\gamma_{c+1}G \cong F / K\gamma_{c+1}F$ together with a homomorphism $F
\to F/K\gamma_{c+1}F$. This finishes our description of the nilpotent
quotient algorithm in the case where $Q = \emptyset$ holds.\smallskip

If, on the other hand, $G$ is given by a finite $L$-presentation
$\la\X\mid\Q\mid\Phi\mid\R\ra$ with $\Q \neq \emptyset$, the algorithm described above applies
to the finitely $L$-presented group $H = \la\X\mid\emptyset\mid\Phi\mid\R\ra$. Write $H =
F/K$ and $G = F/L$ for normal subgroups $K \leq L$.  
The nilpotent quotient algorithm applied to $H$ yields a polycyclic
presentation for $H / \gamma_{c+1}H$ together with a homomorphism $F
\to F/K\gamma_{c+1}F$. This yields
\[
  G / \gamma_{c+1}G \cong F / L\gamma_{c+1}F \cong 
  (F/K\gamma_{c+1}F) / (L\gamma_{c+1}F/K\gamma_{c+1}F).
\]
The subgroup $L\gamma_{c+1}F/K\gamma_{c+1}F$ is finitely generated by
the images of the relations in $\Q$. Standard methods for polycyclic
groups~\cite{Sims94} then give a polycyclic presentation for the
factor group $G / \gamma_{c+1}G$ of the polycyclically presented group
$H/\gamma_{c+1}H$ and a homomorphism $F \to G / \gamma_{c+1}G$.

%%%%%%%%%%%%%%%%%%%%%%%%%%%%%%%%%%%%%%%%%%%%%%%%%%%%%%%%%%%%%%%%%%%%%%%%%%%%
\subsection{Applications of the Nilpotent Quotient Algorithm}\Label{sec:AppsNQL}
The nilpotent quotient algorithm allows us to compute within the lower
central series quotients $G / \gamma_{c+1}G$ of a finitely $L$-presented
group $G$. For instance, it allows us to determine the isomorphism
type of the lower central series sections $\gamma_cG / \gamma_{c+1}G$.
For various self-similar groups, the lower central series sections
often exhibit periodicities. For instance, the Grigorchuk group $\Grig$
satisfies
\def\0{Rozhkov~\cite{Roz96}}
\begin{theorem}[\0]
  The lower central series sections $\gamma_c\Grig/\gamma_{c+1}\Grig$ 
  are $2$-elementary abelian with the following $2$-ranks:
  \[
   \rk_2(\gamma_c\Grig/\gamma_{c+1}\Grig) = \left\{ \begin{array}{cl}
   3\textrm{ or } 2,&\textrm{ if }c=1\textrm{ or }c=2,\textrm{ respectively}\\[0.5ex]
   2,&\textrm{ if }c\in\{2\cdot 2^m+1,\ldots,3\cdot 2^m\}\\[0.5ex]
   1,&\textrm{ if }c\in\{3\cdot 2^m+1,\ldots,4\cdot 2^m\}
   \end{array}\right\}\textrm{ with }m\in\N_0.
  \]
  The group $\Grig$ has finite width $2$.
\end{theorem}
Our implementation of the nilpotent quotient algorithm in~\cite{NQL}
allows a computer algebra system to be applied in the investigation of the
quotients $G / \gamma_cG$ for a finitely $L$-presented group $G$. For instance,
our implementation suggests that the group $\FG_d$ has a maximal
nilpotent quotient whenever $d$ is not a prime-power. Based on this
experimental observation, the following proposition was proved:
\def\0{Bartholdi et al.~\cite{BEH08}}
\begin{proposition}[\0]\Label{prop:MaxNilQuot}
% If $d$ is not a prime-power, then $\FG_d$ has a largest nilpotent
% quotient.
  If $d$ is not a prime-power, the group $\FG_d$ has a maximal nilpotent
  quotient. Its nilpotent quotients are isomorphic to the nilpotent quotients
  of the wreath product $\Z_d \wr \Z_d$.
\end{proposition}
For a prime $p\geq 3$, the lower central series sections $\gamma_c\FG_p
/ \gamma_{c+1}\FG_p$ are $p$-elementary abelian. For $p = 3$, the lower
central series sections $\gamma_c \FG_3 / \gamma_{c+1}\FG_3$ were
computed in~\cite{Bar05}:
\def\0{Bartholdi~\cite{Bar05}}
\begin{proposition}[\0]
  The sections $\gamma_c\FG_3/\gamma_{c+1}\FG_3$ are $3$-elementary 
  abelian with the following $3$-ranks:
  \[  
    \rk_3(\gamma_{c}\FG_3 / \gamma_{c+1}\FG_3)
    = \left\{ \begin{array}{cl} 
       2 \textrm{ or } 1,&\textrm{ if } c = 1 \textrm{ or } c=2\textrm{, respectively}, \\[0.5ex]
       2,& \textrm{ if } c \in \{3^k+2,\ldots, 2\cdot 3^k+1\}, \\[0.5ex]
       1,& \textrm{ if } c \in \{2\cdot 3^k+2,\ldots, 3^{k+1}+1 \}
   \end{array} \right\}
  \]
  with $k \in \N_0$.  The group $\FG_3$ has finite width $2$.
\end{proposition}
For primes $p > 3$, little is known about the series
sections $\gamma_{c}\FG_p / \gamma_{c+1}\FG_p$ so far~\cite{BEH08}. We use the
following abbreviation to list the ranks of these sections: If the
same entry $a \in \N$ appears in $m$ consecutive places in a list,
it is listed once in the form $a^{[m]}$.  The sections $\gamma_c\FG_p /
\gamma_{c+1}\FG_p$ are $p$-elementary abelian. Their $p$-ranks are given
by the following table:
\[
  \begin{array}{cl@{\,}l@{\,}l@{\,}l@{\,}lr}
  \toprule
  p  & \multicolumn{5}{c}{\rk_p \big(\gamma_c\FG_p/
      \gamma_{c+1}\FG_p\big)}&\multicolumn{1}{c}{\textrm{class}}\\
  \midrule
  3 \rule{0ex}{2.5ex} & 2, 1^{[1]}, &2^{[1]}, 1^{[1]}, &2^{[3]},  1^{[3]},&2^{[9]},&1^{[9]}, 
      2^{[27]}, 1^{[27]}, 2^{[65]}                                     &147\\
  5 & 2, 1^{[3]}, &2^{[1]}, 1^{[13]},&2^{[5]}, 1^{[65]},&2^{[25]},&1^{[26]}   &139\\
  7 & 2, 1^{[5]}, &2^{[1]}, 1^{[33]},&2^{[7]}, 1^{[68]} &        &           &115\\
  11& 2, 1^{[9]}, &2^{[1]}, 1^{[97]}& 2^{[4]}        &        &           &112\\
  \bottomrule
  \end{array}
\]
These computational results were obtained with a parallel version of the
nilpotent quotient algorithm in~\cite{BEH08,Har08}. They were intended
to be published in~\cite{EH10}. These computational results extend those
in~\cite{BEH08} significantly so that we obtain detailed conjectures
on the structure of the lower central series sections $\gamma_c\FG_p
/ \gamma_{c+1}\FG_p$: The sections $\gamma_c\FG_p /
\gamma_{c+1}\FG_p$ are $p$-elementary abelian with the following
$p$-ranks: Write $f_p(\ell) = p + (p^2-2p-1)(p^{\ell+1} - 1)/(p-1)$
and $g_p(\ell) = f_p(\ell) + p^{\ell+1}$. Then we conjecture that
\[ 
   \rk_p(\gamma_c\FG_p/\gamma_{c+1}\FG_p) = \left\{\begin{array}{cl}
     2, &\textrm{ if } c \in \{1,p\} \textrm{ or } 
     f_p(\ell) \leq c < g_p(\ell)\textrm{ for some } \ell\in\N_0, \\
     1, &\textrm{ otherwise}
   \end{array}\right.
\]
holds. 
If this conjecture is true, the group $\FG_p$ would have finite width
$2$. For prime powers $3 \leq d \leq 11$, our implementation yields the
following results:
\begin{itemize}\addtolength{\itemsep}{-1ex}
\item For $d = 4$, the Fabrykowski-Gupta group $\FG_4$ satisfies
   \[
      \FG_4 / \FG_4' \cong \Z_4 \times \Z_4\quad\textrm{and}\quad
     \gamma_2\FG_4 / \gamma_3\FG_4 \cong \Z_4.
   \]
   For $3 \leq c \leq 141$, the sections $\gamma_c\FG_4/\gamma_{c+1} \FG_4$ are $2$-elementary
   abelian with $2$-ranks: $2^{[4]}, 3^{[3]}, 2^{[13]}, 3^{[12]}, 2^{[52]}, 3^{[48]}, 2^{[7]}$.
\item For $d = 8$, the Fabrykowski-Gupta group $\FG_8$ satisfies
   \[
     \FG_8 / \FG_8' \cong \Z_8 \times \Z_8,\qquad
     \gamma_2\FG_8 / \gamma_3\FG_8 \cong \Z_8,
   \]
   and
   \[
     \gamma_3\FG_8 / \gamma_4\FG_8 \cong
     \gamma_4\FG_8 / \gamma_5\FG_8 \cong
     \gamma_5\FG_8 / \gamma_6\FG_8 \cong
     \gamma_6\FG_8 / \gamma_7\FG_8 \cong \Z_4.
   \]
   For $7 \leq c\leq 111$, the sections $\gamma_c\FG_8 / \gamma_{c+1}\FG_8$ are $2$-elementary
   abelian with $2$-ranks:
   $2,1,2^{[2]},3,2,3^{[2]},4,3^{[8]},2^{[23]},3^{[5]},2^{[3]},1^{[8]},2^{[16]},3^{[8]},2^{[8]},3^{[16]},4$.
\item For $d = 9$, the Fabrykowski-Gupta group $\FG_9$ satisfies
    \[
      \FG_9 / \FG_9' \cong \Z_9 \times \Z_9,\quad
      \gamma_2\FG_9 / \gamma_3\FG_9 \cong \Z_9,\quad\textrm{and}\quad
      \gamma_3\FG_9 / \gamma_4\FG_9 \cong \Z_9. 
    \]
    For $4\leq c\leq 117$, the sections $\gamma_c\FG_9 /
    \gamma_{c+1}\FG_9$ are $3$-elementary abelian with $3$-ranks: 
    $1^{[5]}, 2^{[6]}, 3,2^{[17]},1^{[38]},1^{[47]}$.
\end{itemize}

%%%%%%%%%%%%%%%%%%%%%%%%%%%%%%%%%%%%%%%%%%%%%%%%%%%%%%%%%%%%%%%%%%%%%%%%%%%%
\section{Computing Dwyer Quotients of the Schur Multiplier}\Label{sec:Dwyer}
The Schur multiplier $M(G)$ of a group $G$ can be defined as the second
homology group $H_2(G,\Z)$ with integer coefficients. It is an invariant
of the group which is of particular interest for infinitely presented
groups because proving the Schur multiplier being infinitely generated
proves that the group does not admit a finite presentation. This is due
to the fact that the Schur multiplier of a finitely presented group is
finitely generated abelian which can be seen as a consequence of Hopf's
formula: If $F$ is a free group and $R \unlhd F$ a normal subgroup so
that $G \cong F/R$ holds, the Schur multiplier $M(G)$ satisfies
\begin{equation}
  M(G) \cong (R \cap F') / [R,F].\Label{eqn:HopfForm}
\end{equation}
However, a group with a finitely generated Schur multiplier is not
necessarily finitely presented~\cite{Bau71}.  For further details on
the Schur multiplier, we refer to~\cite[Chapter~11]{Rob96}.\smallskip

It is known that the Schur multiplier of a finitely $L$-presented
group (and even the Schur multiplier of a finitely presented group) is not
computable in general~\cite{Gor95}. Nevertheless, the Schur multiplier of some
self-similar groups has been computed in~\cite{Gri99,BS10}:
For instance, the Grigorchuk group $\Grig$ satisfies
\def\0{Grigorchuk~\cite{Gri99}}
\begin{proposition}[\0]
  The Schur multiplier \mbox{$M(\Grig)$} is infinitely generated
  $2$-elementary abelian. Therefore, the group $\Grig$ is not finitely
  presented.
\end{proposition}
There are various examples of self-similar groups for which nothing is
known on their Schur multiplier. Even though the Schur multiplier $M(G)$ 
is not computable in general, it is possible to compute successive quotients of $M(G)$
provided that the group $G$ is given by an invariant
finite $L$-presentation~\cite{Har10}. These quotients often exhibit periodicities as
well: For instance, our experiments with the implementation of the algorithm 
in~\cite{Har10} suggest  that the Schur multiplier of
the Fabrykowski-Gupta groups $\FG_d$, for a prime-power $d = p^\ell$, is
infinitely generated. The algorithm for computing successive quotients
of $M(G)$ provides a first method to investigate the structure of the
Schur multiplier of an invariantly finitely $L$-presented group (and
even the Schur multiplier of a finitely presented group).\smallskip

We briefly sketch the idea of this algorithm: Let $G$ be an invariantly
finitely $L$-presented group. Write $G \cong F / K$ for a free group
$F$ and a normal subgroup $K$. Then $G/\gamma_cG \cong F/K\gamma_cF$.
We identify $M(G)$ with $(K \cap F')/[K,F]$ and $M(G/\gamma_cG)$ with
$(K\gamma_cF\cap F')/[K\gamma_cF,F]$ and define
\[
  \varphi_c\colon M(G)\to M(G/\gamma_cG),\: g[K,F]\mapsto g[K\gamma_cF,F].
\]
Then $\varphi_c$ is a homomorphism of abelian groups. In the induction step
of the nilpotent quotient algorithm, the algorithm
computes a homomorphism $F \to F/[K\gamma_cF,F]$. This homomorphism allows
us to compute the image of the Schur multiplier $M(G)$ in $M(G/\gamma_cG)$. In particular, 
it allows us to compute the isomorphism type of the \emph{Dwyer
quotient} $M_c(G) = M(G)/\ker\varphi_c$, for $c \in \N$, where 
\[
  M(G) \geq \ker\varphi_1 \geq \ker\varphi_2 \geq \ldots.
\]
The algorithm for computing $M_c(G)$ has been implemented in 
\Gap. Its implementation allows us to compute the Dwyer
quotients of various self-similar groups: Since the Schur multiplier
of the Grigorchuk group $\Grig$ is $2$-elementary abelian, the Dwyer
quotients of $\Grig$ are $2$-elementary abelian.  We have computed the
Dwyer quotients $M_c(\Grig)$ for $1\leq c\leq 301$. These quotients are
$2$-elementary abelian with the following $2$-ranks:
\[
  1,2,3^{[3]},5^{[6]},7^{[12]}, 9^{[24]}, 11^{[48]}, 13^{[96]}, 15^{[110]}.
\]
These experiments suggest that the Grigorchuk group satisfies
\[
 M_c(\Grig) \cong \left\{\begin{array}{cl}
 \Z_2\textrm{ or }(\Z_2)^2,&\textrm{if }c=1\textrm{ or }c=2\textrm{, respectively},\\
 (\Z_2)^{2m+3},&\textrm{if }c\in\{3\cdot 2^m,\ldots,3\cdot 2^{m+1}-1\},
 \end{array}\right\}
\]
with $m\in\N_0$. For the Fabrykowski-Gupta groups $\FG_d$, 
the algorithm in~\cite{Har10} yields first
insight into the structure of $M(\FG_d)$: We restrict ourself to the
groups $\FG_d$ for prime powers $d = p^\ell$ because,
otherwise, the groups have a maximal nilpotent quotient by Proposition~\ref{prop:MaxNilQuot}. For a 
prime $p \in \{3,5,7,11\}$, the Dwyer quotients $M_c(\FG_p)$ are 
$p$-elementary abelian groups with the following $p$-ranks:
\[
  \begin{array}{cl@{\,}l@{\,}l@{\,}l@{\,}l@{\,}l@{\,}l@{\,}l@{\,}l@{\,}l@{\,}}
    \toprule
    p &\multicolumn{10}{c}{\rk_p(M_c(\FG_p))} \\
    \midrule
    3\rule{0ex}{2.5ex} &0^{[2]},&1^{[3]},& 2^{[0]},& 3^{[9]},& 4^{[1]}, &5^{[26]},& 6^{[4]}, &7^{[77]},& 8^{[13]}, &9^{[12]} \\ % cl = 147
    5&0^{[1]}, &1^{[4]}, &2^{[2]}, &3^{[20]}, &4^{[10]}, &5^{[100]},&6^{[1]}&&& \\ % cl = 138
    7&0^{[1]},& 1^{[2]}, &2^{[6]},& 3^{[2]},& 4^{[14]}, &5^{[42]}, &6^{[14]},&7^{[34]} &&\\ % cl = 115
    11&0^{[1]}, &1^{[2]},&2^{[2]},& 3^{[2]},& 4^{[10]}, &5^{[2]}, &6^{[22]}, &7^{[22]}, &8^{[22]}, &9^{[27]} \\  % cl = 112
    \bottomrule
  \end{array}
\]
As noted by Bartholdi, these experimental results suggest that
\[
  \rk_3(M_{c+1}(\FG_3)) = \left\{
  \begin{array}{cl}
   2\left\lfloor\log_3\left(\frac{2c-1}{10}\right) \right\rfloor + 3, 
   & \textrm{ if }\log_3(2c-1) \in \Z,\\[0.75ex]
    \left\lfloor \log_3(2c-1) \right\rfloor
  + \left\lfloor \log_3\left(\frac{2c-1}{10}\right) \right\rfloor
  + 1, &\textrm{ otherwise,}
  \end{array}\right.
\]
for $c\geq 6$.  Our results for the Dwyer quotients $M_c(\FG_d)$,
for $d\in\{4,8,9\}$, are shown in Table~\ref{tab:Dwyer}
where we list the abelian invariants of $M_c(G)$. Here, a
list $(\alpha_1,\ldots,\alpha_n)$ stands for the abelian group
$\Z_{\alpha_1}\times\cdots\times\Z_{\alpha_n}$. Again, we list the abelian
invariants $(\alpha_1,\ldots,\alpha_n)^{[m]}$ just once if they appear
in $m$ consecutive places.
\begin{table}[ht]
  \begin{center}
  \label{tab:Dwyer}
  \caption{Dwyer quotients of the Fabrykowski-Gupta groups $\FG_d$}\bigskip

  \begin{tabular}{lr}
    \toprule
    \multicolumn{1}{c}{$d$} & \multicolumn{1}{c}{$M_c(\FG_d)$} \\
    \midrule
    & \raisebox{0ex}[2.5ex]{}
    $(1)^{[1]}$ 
    $(2)^{[1]}$ 
    $(2,2)^{[1]}$ 
    $(2,4)^{[4]}$ 
    $(2,2,2,4)^{[1]}$ \\ 4 &
    $(2,2,2,2,4)^{[4]}$  
    $(2,2,2,4,4)^{[16]}$ 
    $(2,2,2,2,4,4)^{[1]}$ 
    $(2,2,2,2,2,4,4)^{[3]}$  \\ &
    $(2,2,2,2,2,2,4,4)^{[16]}$ 
    $(2,2,2,2,2,4,4,4)^{[64]}$ 
    $(2,2,2,2,2,2,4,4,4)^{[5]}$ \\ &
    $(2,2,2,2,2,2,2,4,4,4)^{[11]}$ 
    $(2,2,2,2,2,2,2,2,4,4,4)^{[26]}$ \\
    \midrule
    & \raisebox{0ex}[2.5ex]{}
    $(1) ^ {[1]}$
    $(8) ^ {[2]}$
    $(4,8) ^{[3]}$
    $(2,4,8) ^{[4]}$
    $(2,8,8) ^{[1]}$
    $(2,2,8,8) ^{[2]}$\\ &
    $(2,2,2,8,8) ^{[2]}$
    $(2,2,4,8,8) ^{[2]}$
    $(2,4,4,8,8) ^{[2]}$
    $(2,4,8,8,8) ^{[2]}$\\ \raisebox{1.5ex}[-1.5ex]{8} &
    $(2,8,8,8,8) ^{[8]}$
    $(2,2,8,8,8,8) ^{[4]}$
    $(2,4,8,8,8,8) ^{[20]}$
    $(2,2,4,8,8,8,8) ^{[32]}$\\ &
    $(2,2,8,8,8,8,8) ^{[7]}$
    $(2,2,2,8,8,8,8,8) ^{[16]}$
    $(2,2,2,2,8,8,8,8,8) ^{[16]}$\\ & 
    $(2,2,2,4,8,8,8,8,8) ^{[16]}$
    $(2,2,4,4,8,8,8,8,8) ^{[3]}$\\
    \midrule
    & \raisebox{0ex}[2.5ex]{}
    $(1) ^ {[1]}$
    $(9) ^ {[2]}$
    $(3,9) ^ {[2]}$
    $(3,3,9) ^ {[4]}$
    $(3,9,9) ^ {[2]}$ \\ &
    $(9,9,9) ^ {[2]}$
    $(3,9,9,9) ^ {[2]}$
    $(3,3,9,9,9) ^ {[4]}$
    $(3,9,9,9,9) ^ {[2]}$ \\ \raisebox{1.5ex}[-1.5ex]{9} &
    $(9,9,9,9,9) ^ {[12]}$
    $(3,9,9,9,9,9) ^ {[18]}$
    $(3,3,9,9,9,9,9) ^ {[36]}$ \\ &
    $(3,9,9,9,9,9,9) ^ {[18]}$
    $(9,9,9,9,9,9,9) ^ {[17]}$
    $(3,9,9,9,9,9,9,9) ^ {[12]}$\\
    \bottomrule
  \end{tabular}
  \end{center}
\end{table}
% \begin{itemize}
% \item A standard technique in computational group theory for dealing with
%   non-solvable problems is to deduce algorithms that may allow to compute,
%   say quotient of the object.
% \end{itemize}

%%%%%%%%%%%%%%%%%%%%%%%%%%%%%%%%%%%%%%%%%%%%%%%%%%%%%%%%%%%%%%%%%%%%%%%%%%%%
\section{Coset Enumeration for Finite Index Subgroups}\Label{sec:TC}
A standard algorithm for finitely presented groups is the \emph{coset
enumerator} introduced by Todd and Coxeter~\cite{TC36}. Coset enumeration
is an algorithm that, given a finite generating set of a subgroup $H\leq
G$, computes the index $[G:H]$ provided that this index is finite. Its
overall strategy is to compute a permutation representation for the
group's action on the right-cosets $H \bs G$. For finitely presented
groups, coset enumeration techniques have been investigate for some
time~\cite{Lee63,CDHW73,Neu82,Sims94}. They allow computer algorithms
to be applied in the investigation of finitely presented groups by
their finite index subgroups~\cite{HR94}. It was shown in~\cite{Har11},
that even finitely $L$-presented groups allow one to develop a coset
enumeration process. This latter algorithm reduces the computation to
finite presentations first and then it proves correctness of the obtained
result.  A coset enumerator for finitely $L$-presented groups has various
interesting applications: For instance, it allows one to compute low-index
subgroups, as suggested in~\cite{DSch74}, and it solves the generalized
word problem for finite index subgroups~\cite{Har11}.\smallskip

We briefly sketch the idea of the coset enumeration process in~\cite{Har11}
in the following.
Let $G = \la\X\mid\Q\mid\Phi\mid\R\ra$ be a finitely $L$-presented
group.  Suppose that a subgroup $H \leq G$ is given by its finitely
many generators $\{g_1,\ldots,g_n\}$. We consider the generators
$g_1,\ldots,g_n$ as elements of the free group $F$ over $\X$. Then $E
= \la g_1,\ldots,g_n \ra \leq F$ satisfies $H \cong \UK/K$ where $K
= \la \Q \cup \bigcup_{\sigma\in\Phi^*} \R^\sigma\ra^F$ is the kernel of 
the free presentation. We are to compute the index
$[G:H] = [F:\UK]$. For this purpose, we define $\Phi_\ell =
\{ \sigma \in \Phi^* \mid \|\sigma\| \leq \ell\}$ where $\| \cdot \|$
denotes the usual word-length in the free monoid $\Phi^*$. Consider 
the finitely presented groups $G_\ell = F/K_\ell$ given by the finite presentation
\begin{equation}
  G_\ell = \Big\la \X\:\Big|\: \Q \cup \bigcup_{\sigma \in \Phi_\ell} \R^\sigma\Big\ra.\Label{eqn:FPCov}
\end{equation}
Then $G_\ell$ naturally maps onto $G$ and we obtain a series of subgroups
\[
  \UK_0 \leq \UK_1 \leq \ldots \leq \UK \leq F.
\]
Since $\UK \leq F$ is a finite index subgroup of a finitely generated
group, it is finitely generated by $u_1,\ldots, u_n$, say.  Furthermore,
we have $\UK = \bigcup_{\ell\geq 0} \UK_\ell$. For each $u_i\in \UK$,
there exists $n_i\in\N_0$ so that $u_i \in \UK_{n_i}$. For $m = \max\{ n_i
\mid 1\leq i\leq n\}$ we have $\{u_1,\ldots,u_n\} \subseteq \UK_m$. Thus
$\UK = \UK_m$. In fact, there exists a positive integer $m \in\N_0$
so that $H$ has finite index in the finitely presented group $G_m = \la
\X \mid \Q \cup \bigcup_{\sigma \in \Phi_m} \R^\sigma\ra$.\smallskip

Coset enumeration for finitely presented groups allows us to compute
a permutation representation $\pi\colon F \to \Sym(\UK_m \bs F)$. The
integer $m$ cannot be given \emph{a priori}. However, various coset
enumerators can be applied in parallel to the finitely presented groups
$G_\ell$. In theory, termination is guaranteed for a sufficiently large
integer $\ell$ if $[G:H]$ is finite. Suppose that one coset enumerator has
terminated for $[G_\ell:H]$ and suppose that it has computed a permutation
representation $\pi_\ell\colon F \to \Sym(\UK_\ell\bs F)$.  Then $[G:H]
= [F:\UK]$ divides the index $[G_\ell:H] = [F:\UK_\ell]$. It suffices
to check whether or not $\pi_\ell$ induces a group homomorphism $G
\to \Sym(\UK_\ell\bs F)$. In this case, we obtain $[G_\ell:H] = [G:H]$
and $\pi_\ell$ is a permutation representation for $G$'s action on the
right-cosets $H \bs G$. Otherwise, we have to enlarge the index $\ell$
and we would finally compute the index $[G:H]$ in this way. The following
theorem was proved in~\cite{Har11}:
\begin{theorem}
  For a finitely $L$-presented group $G = \la\X\mid\Q\mid\Phi\mid\R\ra$
  and a homomorphism $\pi\colon F \to H$ into a finite group $H$,
  there exists an algorithm that decides whether or not $\pi$
  induces a group homomorphism $G \to H$.
\end{theorem}
\begin{proof}
  For an explicit algorithm, we refer to~\cite{Har11}.
\end{proof}
Coset enumeration for finitely $L$-presented groups allows various
computations with finite index subgroups; e.g. computing the intersection
of two finite index subgroups, computing the core of a finite index
subgroup, solving the generalized word problem for finite index
subgroups, etc. In the following, we demonstrate the application of our
coset enumerator to the Fabrykowski-Gupta groups $\FG_p$. In
particular, we show how to compute the number of finite index subgroups
with a moderate index.

%%%%%%%%%%%%%%%%%%%%%%%%%%%%%%%%%%%%%%%%%%%%%%%%%%%%%%%%%%%%%%%%%%%%%%%%%%%%
\subsection{An Application of Coset Enumeration: Low-Index Subgroups}\Label{sec:LowX}
As an application of the coset enumeration process, we consider subgroups
with small index in a finitely $L$-presented group. Since the finitely
presented group $G_\ell$ from Eq.~(\ref{eqn:FPCov}) naturally maps onto
the finitely $L$-presented group $G$, it suffices to compute low-index
subgroups of the finitely presented group $G_\ell$. These subgroups
map to subgroups of $G$ with possibly smaller index. On the other hand,
each finite index subgroup of $G$ has a full preimage with same index
in $G_\ell$. Therefore it remains to remove duplicates from the list
of subgroups obtained from the finitely presented group $G_\ell$. For
finitely presented groups, an algorithm for computing all subgroups up to
a given index was described in~\cite{DSch74}. An implementation of this
algorithm can be found in~\cite{LOWX}.  This implementation includes an
algorithm for computing only the normal subgroups of a finitely presented
group~\cite{CD05}. The latter algorithm allows to deal with possibly
larger indices than the usual low-index subgroup algorithms.\smallskip

We first consider the Grigorchuk group $\Grig$: its lattice of
normal subgroups is well-understood~\cite{Bar05,CST01} while its
lattice of finite index subgroups is widely unknown~\cite{Gr05}. It
is known that the Grigorchuk group has seven subgroups of index
two~\cite{Gr05}. In~\cite{Per00}, it was shown that these index-two
subgroups are the only maximal subgroups of $\Grig$. The implementation
of our coset enumeration process allows us to compute the number of
subgroups with index at most $64$ in the group $\Grig$~\cite{Har11}. Our
computations correct the counts in~\cite[Section~7.4]{BGZ03}
and~\cite[Section~4.1]{BG02}. The following list summarizes the number
of subgroups ($\leq$) and the number of normal subgroups ($\unlhd$)
of $\Grig$:
\[
  \begin{array}{cccccccc}
    \toprule
    {\rm index} & 1 & 2 & 4 & 8 & 16 & 32 & 64\\ 
    \midrule
     \leq       & 1 & 7 & 31 & 183 & 1827 & 22931 & 378403 \\
     \unlhd     & 1 & 7 & 7  & 7   & 5    & 3     & 3 \\
    \bottomrule
  \end{array}
\]
For the Fabrykowski-Gupta groups $\FG_p$, where $3 \leq p \leq 11$ is prime,
we only found subgroups with prime-power index in $\FG_p$. Their counts
are as follows:
\[
  \begin{array}{ccccccccc}
   \toprule
     & \multicolumn{2}{c}{p = 3} & \multicolumn{2}{c}{p = 5 } & \multicolumn{2}{c}{p=7} & \multicolumn{2}{c}{p=11} \\
   \raisebox{1.5ex}[-1.5ex]{index} & \leq & \unlhd & \leq & \unlhd & \leq & \unlhd & \leq & \unlhd \\
   \midrule 
   p^0 & 1  & 1 & 1 & 1 & 1 & 1 & 1 & 1 \\
   p^1 & 4  & 4 & 6 & 6 & 8 & 8 & 12&12 \\
   p^2 & 31 & 1 &806& 1 & ? & 1 & ? & 1 \\
   p^3 &1966& 1 & ? & 1 & ? & ? & ? & ? \\
   p^4 & ?  & 4 & ? & ? & ? & ? & ? & ? \\
   p^5 & ?  & 1 & ? & ? & ? & ? & ? & ? \\
   p^6 & ?  & 1 & ? & ? & ? & ? & ? & ? \\
   p^7 & ?  & 4 & ? & ? & ? & ? & ? & ? \\
   \bottomrule
  \end{array}
\]
Here '$?$' denotes an index where our computations did not terminate
within a reasonable amount of time.  The only normal subgroups with index
$p^2$ are the derived subgroups since $\Gamma_p / \Gamma_p' \cong \Z_p
\times \Z_p$ holds~\cite{Gri00}.  For a prime power index $d = p^\ell$,
the groups $\FG_d$ only admit subgroups with prime power index $p^j$:
\[
  \begin{array}{ccccccc}
   \toprule
     & \multicolumn{2}{c}{p^\ell = 2^2} & \multicolumn{2}{c}{p^\ell = 2^3 } & \multicolumn{2}{c}{p^\ell = 3^2} \\
   \raisebox{1.5ex}[-1.5ex]{index} & \leq & \unlhd & \leq & \unlhd & \leq & \unlhd \\
   \midrule 
   p^0 & 1  & 1 & 1 & 1 & 1 & 1 \\
   p^1 & 3  & 3 & 3 & 3 & 4 & 4 \\
   p^2 & 19 & 7 &19 & 7 &76 &13 \\
   p^3 &211 & 7 &163&19 & ? & ? \\
   p^4 &2419&11 &2227&23 & ? & ? \\
%  p^5 &    &   &   &   &   &   \\
%  p^6 &    &   &   &   &   &   \\
%  p^7 &    &   &   &   &   &   \\
   \bottomrule
  \end{array}
\]
For the groups $\FG_6$ and $\FG_{10}$, we obtain the following subgroup counts:
\[
  \begin{array}{ccccc}
    \toprule
    & \multicolumn{2}{c}{\FG_6} & \multicolumn{2}{c}{\FG_{10}} \\
    \raisebox{1.5ex}[-1.5ex]{\rm index} & \leq & \unlhd & \leq & \unlhd \\
    \midrule
    1 &	1  & 1  & 1   & 1 \\
    2 & 3  & 3  & 3   & 3 \\
    3 & 7  & 4  & 0   & 0 \\
    4 & 9  & 1  & 5   & 1 \\
    5 & 0  & 0  & 11  & 6 \\
    6 & 39 & 13 & 0   & 0 \\
    7 & 0  & 0  & 0   & 0 \\
    8 & 45 & 1  & 1   & 1 \\
    9 & 79 & 1  & 0   & 0 \\
    10& 0  & 0  & 113 & 19\\
    \bottomrule
  \end{array} \qquad\qquad\quad
  \begin{array}{ccccc}
    \toprule
    & \multicolumn{2}{c}{\FG_6} & \multicolumn{2}{c}{\FG_{10}} \\
    \raisebox{1.5ex}[-1.5ex]{\rm index} & \leq & \unlhd & \leq & \unlhd \\
    \midrule
    11 & 0   & 0 & 0  &0\\
    12 & 219 & 6 & 0  &0\\
    13 & 0   & 0 & 0  &0\\
    14 & 0   & 0 & 0  &0\\
    15 & 0   & 0 & 0  &0\\
    16 & 188 & 0 & 16 &0\\
    17 & 0   & 0 & 0  &0\\
    18 &1299 & 7 & 0  &0\\
    19 & 0   & 0 & 0  &0\\
    20 & 0   & 0 & ?  &?\\
    \bottomrule
  \end{array}
\]
%\[
%  \begin{array}{cccccccccccccccccccccc}
%    \toprule
%    \multicolumn{2}{c}{\textrm{index}} & 1 & 2 & 3 & 4 & 5 & 6 & 7 & 8 & 9 & 10 & 11 & 12 & 13 & 14 & 15 & 16 & 17 & 18 & 19 & 20\\
%    \midrule
%    & \leq  & 1 & 3& 7& 9& 0&39& 0&45&79& 0&0&219& 0& 0& 0&188 & 0&1299 &0&0\\ \raisebox{1.5ex}[-1.5ex]{$\FG_6$}
%    &\unlhd & 1 & 3& 4& 1& 0&13& 0& 1& 1& 0& 0& 6& 0& 0& 0&  0 & 0& 7   &0&0\\
%    \midrule
%    & \leq  & 1 & 3& 0& 5&11& 0& 0& 1& 0&113&0& 0& 0& 0& 0& 16 & 0& 0 &0& \\ \raisebox{1.5ex}[-1.5ex]{$\FG_{10}$}
%    &\unlhd & 1 & 3& 0& 1& 6& 0& 0& 1& 0&19& 0& 0& 0& 0& 0&  0 & 0& 0 &0& \\
%    \bottomrule
%  \end{array}
%\]

%%%%%%%%%%%%%%%%%%%%%%%%%%%%%%%%%%%%%%%%%%%%%%%%%%%%%%%%%%%%%%%%%%%%%%%%%%%%
\section{Computing Solvable Quotients}\Label{sec:DerSer}
The coset enumeration process in~\cite{Har11} was used to prove the
following version of the Reide\-meister-Schreier theorem for finitely
presented groups in~\cite{Har11b}:
\begin{theorem}\Label{thm:ReidSchrApps}
  Each finite-index subgroup of a finitely $L$-presented
  group is finitely $L$-presented. 
\end{theorem}
\begin{proof}
  For a constructive proof, we refer to~\cite{Har11b}.
\end{proof}
The constructive proof of Theorem~\ref{thm:ReidSchrApps} allows
us to apply the method for finitely $L$-presented groups to finite
index subgroups of a finitely $L$-presented group. As an application
of this method, we consider the successive quotients $G /G^{(i)}$
of the derived series. This series is defined recursively by $G^{(1)} =
G' = [G,G]$ and $G^{(i+1)} = [G^{(i)},G^{(i)}]$ for $i\in\N$. The
isomorphism type of the abelian quotient $G / G'$ can be computed with
the methods from~\cite{BEH08,Har08} provided that $G$ is given by a
finite $L$-presentation. Moreover, it is decidable whether or not $G'$
has finite index in $G$; see~\cite{Har08,BEH08}.\smallskip

Suppose that $G / G'$ is finite. Then the constructive proof
of Theorem~\ref{thm:ReidSchrApps} allows us to compute a finite
$L$-presentation for the finite index subgroup $G'\leq G$. Then we can
compute its abelianization and we can continue this process.  In general,
if $G / G^{(i+1)}$ is finite, we can therefore compute the quotients
$G^{(i+1)} / G^{(i+2)}$ recursively. An alternative approach to compute
the sections $G^{(i)} / G^{(i+1)}$ could generalize the methods for
finitely presented groups~\cite{Lo97}.\smallskip

For the Grigorchuk group $\Grig$, the sections $G^{(i)} / G^{(i+1)}$
of the derived series have been computed by Rozhkov~\cite{Roz93};
see also~\cite{Vie98}:
\def\0{\cite{Roz93}}
\begin{theorem}[Rozhkov~\0]
  The Grigorchuk group $\Grig$ satisfies $[\Grig:\Grig'] = 2^3$,
  \mbox{$[\Grig:\Grig''] = 2^7$}, and $[\Grig:\Grig^{(k)}] = 2^{2+2^{2k-2}}$
  for $k\geq 3$.
\end{theorem}
Our implementation of the Reidemeister-Schreier
Theorem~\ref{thm:ReidSchrApps} yields that
\[
  \Grig/\Grig' \cong (\Z_2)^3,\quad                          % well-known
  \Grig'/\Grig'' \cong \Z_2 \times\Z_2 \times \Z_4,\quad\textrm{and}\quad     % checked with GAP
  \Grig''/\Grig^{(3)} \cong (\Z_2)^2 \times (\Z_4)^3 \times \Z_8.% checked with GAP
\]
Since the abelianization $\FG_p / \FG_p' \cong \Z_p \times \Z_p$
of the Fabrykowski-Gupta group $\FG_p$ is finite~\cite{Gri00},
the derived subgroup $\FG_p'$ satisfies $[\FG_p:\FG_p'] = p^2$. A
finite $L$-presentation for $\FG_p'$ can be computed with the methods
in~\cite{Har11b}. We obtain that
\[
  \FG_3' / \FG_3'' \cong (\Z_3)^2, \qquad
  \FG_3'' / \FG_3^{(3)} \cong (\Z_3)^4,\quad\textrm{and}\quad
  \FG_3^{(3)} / \FG_3^{(4)} \cong (\Z_3)^{10}
\]
as well as $\FG_4' / \FG_4'' \cong (\Z_4)^2$,
\[
  \FG_4'' / \FG_4^{(3)} \cong \Z_2 \times (\Z_4)^2 \times \Z_8,\quad\textrm{ and }\quad
  \FG_4^{(3)} / \FG_4^{(4)} \cong (\Z_2)^3\times (\Z_4)^9\times (\Z_8)^3.
\]
%\[
%  \begin{array}{ccccc}
%  \toprule
%    & p = 3 & p = 5 & p = 7 & p = 11 \\
%  \midrule
%  \Gamma_p/\Gamma_p'   & \Z_3 \times \Z_3 & \Z_5\times \Z_5 & \Z_7 \times \Z_7 & \Z_{11} \times \Z_{11} \\[0.75ex]
%  \Gamma_p'/\Gamma_p'' & (\Z_3)^4 & (\Z_5)^4 & (\Z_7)^6 & (\Z_{11})^{10}                                \\[0.75ex]
%  \Gamma_p''/\Gamma_p''' & (\Z_3)^{10} &          &          &                                          \\[0.75ex]
%  \bottomrule
%  \end{array}
%\]
%\[
%  \begin{array}{cccc}
%  \toprule
%   d  & \Gamma_d/\Gamma_d' & \Gamma_d'/\Gamma_d'' & \Gamma_d''/\Gamma_d''' \\
%  \midrule
%  3   & (\Z_3)^2     &  (\Z_3)^4 & (\Z_3)^{10} \\
%  % ./LowNormalFG/Composite/DerQuot/Der4.log
%  4   & (\Z_4)^2     &  \Z_2 \times (\Z_4)^2 \times \Z_8 & (\Z_2)^3\times (\Z_4)^9\times (\Z_8)^3\\ 
%  5   & (\Z_5)^2     & (\Z_5)^4       &  \\
%  6   & (\Z_6)^2     & (\Z_6)^5       & \\
%  7   & (\Z_7)^2     & (\Z_7)^6       & \\ 
%  8   & (\Z_8)^2     & (\Z_8)^7       & \\
%  9   & (\Z_9)^2     & (\Z_9)^8       & \\
%  10  & (\Z_{10})^2  & (\Z_{10})^9    & \\
%  11   & (\Z_{11})^2 & (\Z_{11})^{10} & \\
%  \bottomrule
%  \end{array}
%\]
For $5 \leq d \leq 41$, our computations suggest the following 
\begin{proposition}
  For $d \geq 5$, $\FG_d$ satisfies $\FG_d
  / \FG_d' \cong (\Z_d)^2$ and $\FG_d' / \FG_d'' \cong
  (\Z_d)^{d-1}$.
\end{proposition}
\begin{proof}
  It was already shown in~\cite{Gri00} that $\FG_d / \FG_d' \cong
  \Z_d \times \Z_d$ holds. For the second statement, we combine the methods
  from~\cite{FAZR11} and~\cite{Gri00}: For primes $p$, the structure
  of the congruence subgroups $\FG_p / \Stab_{\FG_p}(n)$, $n\in\N$,
  were studied in~\cite{FAZR11}. Moreover, it was shown in~\cite{Gri00} that,
  for $d \geq 5$, the index $[\FG_d':\FG_d'']$ is finite.\smallskip

  Let $d \geq 5$ be given. Denote by $\Stab_{\FG_d}(1)$ the first level
  stabilizer in $\FG_d$.  Then $\FG_d = \Stab_{\FG_d}(1) \rtimes \la a
  \ra$ and $\Stab_{\FG_d}(1) = \la r,r^a,\ldots,r^{a^{d-1}} \ra$ hold. Since
  $\FG_d' = \la [a,r]\ra^{\FG_d} = \la r^{-a}\,r\ra^{\FG_d}$,
  we have that $\FG_d' \leq \Stab_{\FG_d}(1)$ and, as $\FG_d / \FG_d'
  \cong \Z_d \times \Z_d$ holds, we have that $[\Stab_{\FG_d}(1):\FG_d']
  = d$. More precisely, we have $\Stab_{\FG_d}(1) = \FG_d' \rtimes \la
  r \ra$.\smallskip

  For each $0\leq i < d$, we write $g_i = r^{a^i}$. In the following, indices
  are read modulo $d$. For $0\leq\ell < d$, $g_i^\ell$ decomposes as
  $(1,\ldots,1,r^\ell,a^\ell,1,\ldots,1)$ where $a^\ell$ is at position
  $i$. If \mbox{$|\ell-k|>1$}, the commutator $[g_i^\ell,g_j^k]$ is
  trivial; otherwise, the commutator $[g_i^\ell,g_{i+1}^k]$ decomposes
  as $(1,\ldots,1,[a^\ell,r^k],1,\ldots,1)$ with $[a^\ell,r^k]$ at
  position $i$. Since $[a^\ell,r^k] \in \Stab_{\FG_d}(1)$, we have that
  $[g_i^\ell,g_j^k] \in \Stab_{\FG_d}(2)$. Thus, $\Stab_{\FG_d}(1)
  / \Stab_{\FG_d}(2)$ is abelian and it is generated by the images
  of the elements $g_0,\ldots,g_{d-1}$. Because $[a^\ell,r^k] =
  a^{-\ell}\,r^{-k}\,a^\ell\,r^k = g_\ell^{-k} g_0^k$, we have that
  $[g_i^\ell,g_j^k] \in \Stab_{\FG_d}(3)$ if and only if $\ell\,k \equiv
  0 \pmod d$. Therefore $\Stab_{\FG_d}(1) / \Stab_{\FG_d}(2) \cong \Z_d
  \times \cdots \times \Z_d$ and $\FG_d / \Stab_{\FG_d}(2) \cong \Z_d \wr
  \Z_d$. Since $\Stab_{\FG_d}(1) / \Stab_{\FG_d}(2)$ is abelian,
  we have that $\Stab_{\FG_d}(1)' \leq \Stab_{\FG_d}(2)$. Because each
  generator of $\Stab_{\FG_d}(1)$ has order $d$, the largest abelian
  quotient $\Stab_{\FG_d}(1) / \Stab_{\FG_d}(1)'$ has order at most
  $d^d$. It follows that $\Stab_{\FG_d}(2) = \Stab_{\FG_d}(1)'$. Moreover,
  we have $\Stab_{\FG_d}(2) = \Stab_{\FG_d}(1)' \leq \FG_d'$ and, since
  $\FG_d' \leq \Stab_{\FG_d}(1)$ holds, it follows that $\FG_d'' \leq
  \Stab_{\FG_d}(1)' = \Stab_{\FG_d}(2)$. The proofs in~\cite{Gri00,BEH08}
  yield that $\Stab_{\FG_d}(2) \leq \FG_d''$ if $d \geq 5$.  Therefore
  $d^{d-1} = |\FG_d' / \Stab_{\FG_d}(2)| = |\FG_d' / \FG_d''|$ and $\FG_d'
  / \FG_d'' \cong \Z_d \times \cdots \times \Z_d$.
\end{proof}
The constructive proof of Theorem~\ref{thm:ReidSchrApps} in~\cite{Har11b}
yields a finite $L$-presentation over the Schreier generators of
the subgroup. By the Nielsen-Schreier theorem (as, for instance,
in~\cite[6.1.1]{Rob96}), a subgroup $H$ with index $m = [G:H]$ in an
$n$-generated finitely $L$-presented group $G$ has $nm+1-m$ Schreier
generators. The Fabrykowski-Gupta groups are $2$-generated and therefore,
the subgroup $\FG_3^{(3)}$ satisfies \mbox{$[\FG_3:\FG_3^{(3)}] =
3^{16}$}.  Thus $\FG_3^{(3)}$ has $3^{16}-1$ Schreier generators as a
subgroup of the $2$-generated group $\FG_3$. Therefore, computing the
sections $\FG_3^{(i)} / \FG_3^{(i+1)}$, $i\geq 4$, with the above method
is hard in practice.

\subsection*{Acknowledgments}
I am grateful to Laurent Bartholdi for valuable comments and
suggestions.

\def\cprime{$'$}

\noindent Ren\'e Hartung,
{\scshape Mathematisches Institut},
{\scshape Georg-August Universit\"at zu G\"ottingen},
{\scshape Bunsenstra\ss e 3--5},
{\scshape 37073 G\"ottingen},
{\scshape Germany}\\[1ex]
{\it Email:} \qquad \verb|rhartung@uni-math.gwdg.de|\\[2ex]

\end{document}